\input ./style/arxiv-general.cfg
\documentclass[aop,MSNbibl,dvips]{arximspdf}
\makeatletter
   \@ifpackageloaded{graphicx}{}{\usepackage{graphicx}}
\makeatother

\doi{10.1214/14-AOP915} 
\volume{43}
\issue{4}
\pubyear{2015}
\firstpage{1643}
\lastpage{1662}

\makeatletter
\def\mod{\ \mathrm{mod}\ }
\newcommand{\eqref}[1]{(\ref{#1})}
\newtheorem{thmm}{Theorem}
\newtheorem{cor}{Corollary}
\newtheorem{theorem}[thmm]{Theorem}
\newtheorem{lem}{Lemma}
\newtheorem{lemma}[lem]{Lemma}
\newtheorem{prop}{Proposition}
\newtheorem{proposition}[prop]{Proposition}

\newtheorem{conjecture}{Conjecture}
\newtheorem{problem}{Problem}
\newproclaim{example}{Example}
\newproclaim{remark}{Remark}

\def\tG{{\tilde G}}
\def\Gr{\operatorname{Gr}}
\def\Inv{\operatorname{Inv}}
\def\Z{\mathbb{Z}}
\def\id{\mathrm{id}}
\def\af{\mathrm{af}}
\def\tA{{\tilde A}}
\def\tF{{\tilde F}}
\def\tH{\tilde H}
\def\tY{\tilde Y}
\def\tX{\tilde X}
\def\tS{\tilde S}
\def\deg{\operatorname{deg}}

\newcommand
\defn[1]{\textit{#1}}
\def\Prob{\operatorname{Prob}}
\def\sp{\operatorname{span}}
\def\pp{{\mathbf{p}}}
\def\i{{\mathbf{i}}}
\def\tW{{\hat{W}}}
\def\ha{{\hat{\alpha}}}
\makeatother

\begin{document}
\begin{frontmatter}

\title{The shape of a random affine Weyl group element and
random core partitions}
\runtitle{Random affine Weyl group elements}

\begin{aug}
\author[A]{\fnms{Thomas}~\snm{Lam}\corref{}\ead[label=e1]{tfylam@umich.edu}\thanksref{T1}}
\thankstext{T1}{Supported by
NSF Grants DMS-06-52641 and DMS-09-01111, and by a Sloan Fellowship.}
\runauthor{T. Lam}
\affiliation{University of Michigan}
\address[A]{Department of Mathematics\\
University of Michigan\\
Ann Arbor, Michigan 48109\\
USA\\
\printead{e1}} 
\end{aug}

\received{\smonth{12} \syear{2012}}
\revised{\smonth{1} \syear{2014}}

%
\begin{abstract}
Let $W$ be a finite Weyl group and $\tW$ be the corresponding affine
Weyl group. We show that a large element in $\tW$, randomly generated
by (reduced) multiplication by simple generators, almost surely has one
of $|W|$-specific shapes. Equivalently, a reduced random walk in the
regions of the affine Coxeter arrangement asymptotically approaches one
of $|W|$-many directions. The coordinates of this direction, together
with the probabilities of each direction can be calculated via a Markov
chain on $W$.

Our results, applied to type $\tilde A_{n-1}$, show that a large random
$n$-core obtained from the natural growth process has a limiting shape
which is a piecewise-linear graph. In this case, our random process is
a periodic analogue of TASEP, and our limiting shapes can be compared
with Rost's theorem on the limiting shape of TASEP.
\end{abstract}

%

\end{frontmatter}

\section{Introduction}
Let $W$ denote a finite Weyl group with root system $R$, and let $\tW$
denote the corresponding affine Weyl group, acting on a real vector
space~$V$. They are the most important and classical reflection groups.

\subsection{Random walks in the affine Coxeter arrangement}
The affine Coxeter arrangement of $W$ gives a regular tessellation of
$V$. Define a random walk $X = (X_0,X_1,\ldots)$ in the alcoves, called
the \textit{reduced random walk}. We start at the fundamental alcove and
at each step we cross one adjacent hyperplane chosen uniformly at
random, subject to the condition that we never cross a hyperplane
twice. See Figure~\ref{fig:walk}.

\begin{figure}

\includegraphics{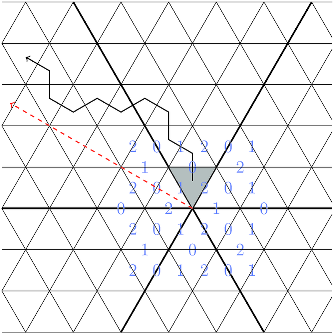}

\caption{A reduced random walk in the alcoves of the $\tilde A_2$
arrangement. The shown walk has reduced word $\cdots1020120210$. The
random walk will almost surely be asymptotically parallel to the red
dashed line. The thick lines divide $V$ into Weyl chambers.}
\label{fig:walk}
\end{figure}

This process is a transient Markov chain. More algebraically, it is
equivalent to a random infinite reduced word for $\tW$ obtained by
multiplying by simple generators one at a time, subject to the
condition that the length increases. Nonrandom infinite reduced words
in the affine Weyl group have a beautiful structure theory, which we
recently studied in relation to factorizations in loop groups \cite{LP}.
We prove here the following.

\begin{theorem}\label{T:main}
Let $(X_0,X_1,\ldots)$ be a reduced random walk in $\tW$. There exists
a unit vector $\psi\in V$ so that almost surely we have
%
\begin{equation}
\label{E:main} \lim_{N \to\infty}v(X_N) \in W \cdot\psi,
\end{equation}
where $v(X_i)$ denotes the unit vector pointing toward the central
point of $X_i$.
\end{theorem}

Thus the reduced walk has one of finitely many asymptotic directions.
The random walk we study here is different to the walks on hyperplane
arrangements that we have seen in the literature; see for example \cite
{BHR,BD}.


\subsection{A remarkable Markov chain on $W$}
In Section~\ref{ss:W}, we define a Markov chain on the finite Weyl
group $W$. Roughly speaking, this Markov chain is obtained by
projecting the affine Grassmannian weak order onto $W$. Unlike the
reduced random walk on $\tW$, this Markov chain is irreducible and
aperiodic (Proposition~\ref{P:sca}), and thus has a unique invariant
distribution $\{\zeta(w) \mid w \in W\}$.

The vectors $W \cdot\psi$ lie in different Weyl chambers $C_w$, and we
let $X \in C_w$ denote the event that the reduced random walk $X$
eventually stays in $C_w$. The probabilities $\Prob(X \in C_w)$ vary
depending on $w$: in $\tilde A_4$, one Weyl chamber is 96 times more
likely than another.
The root system notation of the next theorem is reviewed in
Section~\ref{ss:notation}.

\begin{theorem}\label{T:main2}
The vector $\psi$ of Theorem~\ref{T:main} is given by
\[
\psi= \frac{1}{Z}\sum_{w \in W \dvtx r_\theta w > w} \zeta(w)
w^{-1}\bigl(\theta^\vee\bigr),
\]
where $\theta$ is the highest root of $W$ and $Z$ is a normalization
factor. Furthermore,
\[
\Prob(X \in C_w) = \zeta\bigl(w^{-1}w_0\bigr).
\]
%
\end{theorem}

Thus, the invariant distribution $\zeta$ determines two apparently
unrelated quantities: the coordinates of the asymptotic directions, and
the probabilities of each direction. This surprising duality is
ultimately related to the associativity of the Demazure or monoidal
product in a Coxeter group. In Section~\ref{ss:Shi}, we give an
alternative formula for $\zeta(w)$, expressed as a calculation
involving a sum over the regions of the \textit{Shi arrangement} of $W$.
We also conjecture (Conjecture \ref{conj:A}) that in type $A$ the point
$\psi$ of Theorem~\ref{T:main} is in the same direction as $\rho
^\vee$.
In joint work with Williams \cite{LW}, we conjecture that a
multivariate generalization of this Markov chain on the symmetric group
has remarkable Schubert positivity properties. Some of these
conjectures have been established by Ayyer and Linusson \cite{AL} and
Linusson and Martin \cite{LM}.

\subsection{Random $n$-core partitions}
In the case of $W = A_{n-1}$, Theorem~\ref{T:main} applied to a random
reduced walk conditioned to remain in the fundamental Weyl chamber can
be interpreted in terms of $n$-core partitions. Recall that a Young
diagram is an $n$-core if no $n$-ribbon can be removed from it. Grow a
random $n$-core from the empty partition by randomly adding boxes to
the Young diagram, subject to the condition that the shape is always an
$n$-core. The notation in the following theorem is explained in
Section~\ref{sec:cores}.

\begin{theorem}\label{T:cores}
For each $n$, there exists a piecewise-linear curve $C_n$, so that for
each $\varepsilon,\delta> 0$, there exists an $M$ such that for every $N
> M$, we have
\[
\Prob \bigl(\bigl|D\bigl(\lambda^{(N)}\bigr) -C\bigr| > \delta \bigr) <
\varepsilon,
\]
where $D(\lambda^{(N)})$ is the diagram of a random $n$-core of degree $N$.
\end{theorem}

\begin{figure}

\includegraphics{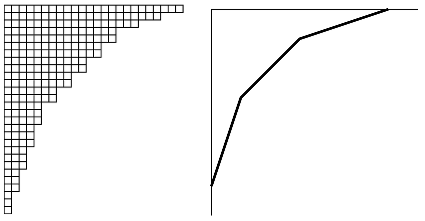}

\caption{A large random $4$-core, and the piecewise-linear curve $C_4$.}
\label{fig:core}
\end{figure}

%

Conjecture \ref{conj:A} (verified for $n
\leq6$)\setcounter{footnote}{1}\footnote{Ayyer
and Linusson \cite{AL2} have reported that they have established
Conjecture \ref{conj:A}.} gives explicit coordinates for the curve $C_n$ (see Figure \ref{fig:core}).

There is a growth model on partitions naturally obtained from TASEP on
the integer lattice \cite{Joh,Ros}, where initially the negative
integers are all occupied by balls/particles and the nonnegative
integers are all vacant. The\vadjust{\goodbreak} particles jump toward the right into
adjacent vacant spaces. Our growth process on $n$-cores corresponds to
a periodic analogue of TASEP: now particles that are distance $n$ apart
are conditioned to jump together. As explained in Section~\ref
{sec:cores}, after appropriate scaling (and assuming Conjecture \ref
{conj:A}), the limit curve $C_n$ of Theorem~\ref{T:cores} approaches,
in the limit $n \to\infty$, the degree 2 curve which is the limit
shape of TASEP with exponential waiting time \cite{Ros}.

\subsection{(Co)homology of the affine Grassmannian}
In this project, we were initially motivated by the study of families
of symmetric functions which represent Schubert classes in the
($K$)-cohomology of the affine Grassmannian $\Gr_{\mathrm{SL}(n)}$ of $\mathrm{SL}(n)$
\cite{LamSP,LSS}. These symmetric functions, called $k$-Schur functions
and affine Stanley symmetric functions, are ``affine'' analogues of
Schur functions, the latter playing a key role in the theory of
Schur-measure and Plancherel-measure random partitions. In a similar
manner, the symmetric functions mentioned above give rise to
Plancherel-like measures on $n$-cores. These measures are however
distinct from the random growth processes studied in this paper.


Instead, our main result may have an interpretation in terms of large
products $\xi^N \in K_*(\Gr_{\mathrm{SL}(n)})$ of an element $\xi$ in the
$K$-homology of the affine Grassmannian---it describes the asymptotics
of the ``spreading out'' over the affine Grassmannian of products of
this class under the Pontryagin multiplication of a loop group (see
Section~\ref{ss:Grass}).

This connection to the infinite-dimensional geometry of $\Gr_{\mathrm{SL}(n)}$
has concrete probabilistic consequences: in a separate article, we plan
to apply this geometry to the calculation of the boundary of the affine
Grassmannian weak order.

\section{Walks in the affine Coxeter arrangement and reduced words}

\subsection{Affine Weyl groups}
\label{ss:notation}
For affine Weyl groups, we use the references \cite{Hum,Kac}.

We denote the simple generators of $W$ by $\{s_i \mid i \in I\}$ and by
$w_0$ the longest element of $W$. Let $s_0$ be the additional simple
generator of $\tW$. The Weyl\vadjust{\goodbreak} group acts as linear reflections in a real
vector space $V$, and the affine Weyl group act as affine reflections
in $V$. We let $\ell\dvtx W \to\Z$ and $\ell\dvtx \tW\to\Z$ denote the length
functions.

We let $R \subset V^*$ denote the set of roots of $W$, and let $R = R^+
\sqcup R^-$ denote the decomposition into positive and negative roots.
The set $R_\af$ of affine roots consists of the elements $\{\alpha+ n
\delta\mid\alpha\in R \mbox{ and } n \in\Z\} \cup\{n\delta
\mid n \in\Z-\{0\}\}$. The roots $\ha= \alpha+ n\delta$ are the real
affine roots, and $\ha$ is positive (resp., negative) if and only if
either (a) $\alpha\in R^+$ and $n \geq0$ (resp., $\alpha\in R^+$ and
$n < 0$), or (b) $\alpha\in R^-$ and $n > 0$ (resp., $\alpha\in R^-$
and $n \leq0$). We denote the positive affine roots by $R_\af^+$ and
the negative affine roots by $R_\af^-$. The simple roots are denoted
$\{
\alpha_i \mid i \in I \cup\{0\}\}$, and we have $\alpha_0 = \delta
-\theta$, where $\theta$ is the highest root. We let $r_\theta$ denote
the reflection in the hyperplane perpendicular to $\theta$.

To each real affine root $\ha= \alpha+ k \delta$, we associate the
(affine) hyperplane $H_\ha= H_\alpha^k = \{v \in V \mid\langle
v,\alpha \rangle
=-k\}$. The \textit{affine Coxeter arrangement} is the hyperplane
arrangement consisting of all such $H_\ha$. We also associate to each
real affine root $\ha$ a coroot $\ha^\vee$. The connected components of
the complement of affine Coxeter arrangement are known as alcoves. The
fundamental alcove $A^\circ$ is bounded by the hyperplanes
corresponding to the simple roots. There is a bijection $x \mapsto A_x$
between the alcoves and $\tW$, and we shall pick conventions so that
$A_{s_ix}$ and $A_x$ are adjacent, separated by the hyperplane
corresponding to $x^{-1} \cdot\alpha_i$. The \textit{Weyl chambers} are
the connected components of the complement to the finite Coxeter
arrangement, where only the $H_\alpha$'s are used for $\alpha\in R$.
The \textit{fundamental chamber} is the Weyl chamber containing the
fundamental alcove. Affine Weyl group elements corresponding to alcoves
inside the fundamental chamber are called \textit{affine Grassmannian}. We
shall also need the right action $w\dvtx A_x \mapsto A_{xw^{-1}}$ of $W$ on
the set of alcoves. The right action of $w^{-1}$ takes the fundamental
chamber to the Weyl chamber $C_w$ labeled by $w$ (the one containing
the alcove $A_w$). The elements in $C_w$ are of the form $xw$, where
$x$ is an affine Grassmannian element.

There is an isomorphism $\tW= W\times Q^\vee$, where $Q^\vee$ denotes
the coroot lattice of $W$. If $\lambda\in Q^\vee$, we denote by
$t_\lambda\in\tW$ the corresponding element in $\tW$, called a
translation element. For $x =w t_\lambda\in\tW$, we have
%
\begin{equation}
\label{E:affineaction} w t_\lambda\cdot(\alpha+ n\delta) = w\alpha+ \bigl(n-\langle
\lambda,\alpha \rangle\bigr)\delta.
\end{equation}
The inversions $\Inv(x) \subset R_\af^+$ of $x$ are exactly the real
affine roots which are sent to negative roots. Equivalently, $\Inv(x)$
consists of the roots corresponding to hyperplanes separating $A_x$
from $A^\circ$. Note that with these conventions, $A_{t_\lambda}$ is
obtained from $A^\circ$ by translation by the vector $-\lambda$. The
\textit{left weak order} on $\tW$ is given by $x \preceq x'$ if and only
if $\Inv(x) \subseteq\Inv(x')$. We shall also write $A \preceq A'$ for
the weak order applied to alcoves, and write $A \lessdot A'$ for the
cover relations. We say that an alcove $A$ is \textit{of type $w$} if $A =
A_{w t_\lambda}$.

Let $\rho= \frac{1}{2}\sum_{\alpha\in R^+} \alpha$ be the half-sum of
positive roots. Recall that $\lambda\in Q^\vee$ is antidominant if
$\langle\lambda,\alpha \rangle\leq0$ for $\alpha\in R^+$. The
following result is
standard \cite{LLMS,LamSP}.

\begin{lem}\label{lem:Grass}
Suppose $x = w t_\lambda$. Then $x$ is affine Grassmannian if and only if
$\lambda$ is antidominant and for every $\alpha\in R^+$ such that $w
\alpha
\in R^-$ we have $\langle\lambda,\alpha \rangle < 0$. We then have
$\ell(x) =
-\langle\lambda,2\rho \rangle - \ell(w)$.
\end{lem}
%

\subsection{The reduced random walk on alcoves}\label{ss:random}
We define a random walk on alcoves. The walk begins at $X_0 = A^\circ$.
Given $(X_0,X_1,\ldots,X_\ell)$, we pick $X_{\ell+1}$ uniformly at
random among the alcoves adjacent to (i.e., sharing a facet with)
$X_\ell$, with the constraint that the hyperplane separating $X_\ell$
and $X_{\ell+1}$ has not been crossed previously. It follows easily
from Coxeter group theory that such walks can never ``get stuck.''

Based on the definition, somewhat surprisingly we get:

\begin{lemma}
The process $(X_0,X_1,\ldots)$ is a Markov chain.
\end{lemma}

\begin{pf}
The hyperplanes that have been crossed during the first $\ell$ steps of
the walk $(X_0,X_1,\ldots,X_\ell)$ are exactly the hyperplanes
separating $X_\ell$ from $X_0 = A^\circ$.
\end{pf}

We call this process the \emph{random walk in $\tW$} (or sometimes the
\emph{reduced random walk in $\tW$}), starting at the fundamental
alcove. We shall also consider the process $(Y_0,Y_1,\ldots)$ where the
random walk is constrained to stay within the fundamental Weyl chamber.
We call this the \textit{reduced affine Grassmannian random walk} in~$\tW$.

\subsection{Reformulation in terms of infinite reduced words}
An infinite reduced word $\i= \cdots i_3 i_2 i_1$ is an infinite word
such that $i_r i_{r-1} \cdots i_1$ is a reduced word for $\tW$, for any
$r$. The Coxeter-equivalence of reduced words can be extended to \textit{braid limits} of infinite reduced words. It is known that any infinite
reduced word $\i$ of $\tW$ is braid equivalent to an infinite reduced
word of the form $\cdots\tau\tau\tau u$, where $\tau$ is the
reduced word of a translation element, and $u$ is a finite reduced word
for $\tW$ (see~\cite{Ito,LP}).
%

Sequences $(X_0,X_1,\ldots)$ of alcoves as considered in Section~\ref
{ss:random} are tautologically in bijection with infinite reduced
words. Thus, Theorem~\ref{T:main} says that a random infinite reduced
word $\i$ is not only almost surely braid equivalent to $\tau^\infty$
for one of $|W|$-many $\tau$'s, but indeed that almost surely $\i$ and
$\tau^\infty$ asymptotically converge to the same point of the boundary
of the Tits cone (cf. \cite{LP}, Remark~4.5).

\section{Projection to the finite Weyl group}
\subsection{A Markov chain on $W$}\label{ss:W}
We define a Markov chain with finite state space~$W$, which appears to
be of independent combinatorial interest. Let\vadjust{\goodbreak} $r = |I|+1$ be the rank
of~$\tW$. The transition probability from $w$ to $v$ is given by
\[
p_{w,v} = %
\cases{ 1/r, & \quad$\mbox{if $v = s_i w$
and $\ell(v) < \ell(w)$}$, \vspace *{2pt}
\cr
1/r, &\quad $\mbox{if $v = r_\theta
w$ and $\ell(v) > \ell(w)$}$, \vspace *{2pt}
\cr
k/r, & \quad $\mbox{if $v = w$}$,
\vspace*{2pt}
\cr
0, &\quad  $\mbox{otherwise}$,} %
\]
where $k$ is chosen so that $\sum_{v \in W} p_{w,v} = 1$. Let $P =
(p_{w,v})$ denote the transition matrix. Let $\Theta_W$ denote the
directed graph on $W$ with edges given by the nonzero transitions
(see Figure \ref{fig:S3}).
Let
$Z_0,Z_1,\ldots$ be the Markov chain on $\Theta_W$ with transition
matrix $P$.

\begin{figure}

\includegraphics{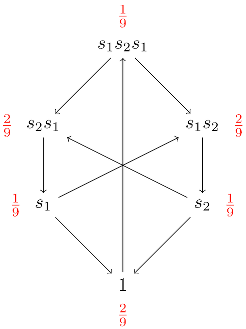}

\caption{The graph $\Theta_{S_3}$ (with the transitions from a vertex
to itself removed) and the stationary distribution $\zeta_{S_3}$.}
\label{fig:S3}
\end{figure}

%

\begin{prop}\label{P:sca}
The Markov chain $(Z_0,Z_1,\ldots)$ is irreducible and aperiodic.
\end{prop}

\begin{pf}
Aperiodicity is clear from the definition. Strong connectedness follows
from \cite{HST}, Theorem~4.2.
\end{pf}

It follows that $(Z_0,Z_1,\ldots)$ has a unique limit stationary distribution.

\begin{problem}
Explicitly describe the stationary distribution $\zeta= \zeta_W$ of
$(Z_0,Z_1,\ldots)$ for each $W$.
\end{problem}

This distribution appears to have remarkable enumerative properties,
especially for the symmetric group \cite{LW}.

\begin{conjecture}\label{conj:Aeigenvector}
Let $W = S_n$. Then $\zeta(w)/\zeta(w_0)$ is an integer for all $w
\in
W$, and $\zeta(1)/\zeta(w_0) = \prod_{k=0}^{n-1} {n\choose k} = \max_{w
\in W}( \zeta(w)/\zeta(w_0))$.\footnote{After this paper was written,
Svante Linusson pointed out to us that the integrality part of
Conjecture \ref{conj:Aeigenvector} follows from the work of Ferrari and
Martin \cite{FM} on multitype TASEP. Aas \cite{Aas} has announced
a proof of the product expression for $\zeta(1)/\zeta(w_0)$.}
\end{conjecture}

\begin{remark}\label{rem:othertypes3}
The integrality part of Conjecture \ref{conj:Aeigenvector} fails for
other types. For example, it is false for $W$ of type $B_3$.
However, the weighted version of $\Theta_W$, as described in Remark~\ref
{rem:othertypes} and Section~\ref{ss:Grass}, still appears to retain
these properties.
\end{remark}

\begin{remark}\label{rem:reduced_words}
Let $\mu_N$ be the probability measure on length $N$ elements of $\tW$,
where $\mu_N(x)$ is proportional to the number of reduced words of $x$.
Define $P'$ by setting the diagonal entries of $P$ to 0. The matrix
$P'$ is a sub-stochastic matrix, which nevertheless calculates the
projected measures $\pi(\mu_N)$ after scaling. (The matrix $P'$ weights
each path equally regardless of the valency of the vertices that it
passes through.)

After scaling, and conjugation by a suitable diagonal matrix $D$, one
does obtain a Markov chain with transition matrix given by $Q = r
D^{-1} P' D$. The methods in this section will still prove Corollary~\ref{C:main} for the measures $\mu_N$ (but with a different limit
$\psi$).
\end{remark}

\subsection{Projection}\label{ss:proj}
Let $(Y_0,Y_1,\ldots)$ denote the affine Grassmannian random walk of
\ref{ss:random}. We let $(\tY_0,\tY_1,\ldots)$ denote the delayed
random walk, where $\tY_{i+1}$ has probability $k/r$ of being equal to
$\tY$, where $r = |I|+1$ is the rank of the affine Weyl group, and $k$
is the number of facets of $\tY_i$ which separate $\tY_i$ from
$A^\circ
$. Each of the transitions in the original random walk now have
probability $1/r$. Similarly, define~$\tX$.

Let $\pi\dvtx  \tW\to W$ be the projection given by $wt_\lambda\mapsto w$.
The following proposition is a key observation of the paper.

\begin{prop}\label{P:Markov}
The projection $\pi(\tY_0,\tY_1,\ldots)$ of the delayed affine
Grassmannian random walk is the Markov chain $(Z_0,Z_1,\ldots)$, with
initial condition $Z_0 = \id$.
\end{prop}

The result follows from Lemmas \ref{lem:Grass} and~\ref{L:s0}.

\begin{lem}\label{lem:high}
Let $\alpha\in R^+ - \{\theta\}$. Then $\langle\theta^\vee,\alpha
\rangle
\in\{0,1\}$.
\end{lem}

\begin{pf}
The sum $\alpha- k \theta$ can be a root only if $k \in\{0,1\}$.
\end{pf}

\begin{lem}\label{L:s0}
Suppose $x = w t_\lambda\in W_{\mathrm{af}}$ is affine Grassmannian. Then\break
$\ell(r_\theta w)
> \ell(w)$ in $W$ if and only if $s_0 x$ is affine Grassmannian and
$s_0 x \succ
x$.
\end{lem}

\begin{pf}
Suppose that $\ell(r_\theta w) > \ell(w)$. Let $\alpha= w^{-1}
\theta
\in R^+$. To show that $s_0x \succ x$, we compute
\[
x^{-1} \alpha_0 = t_{-\lambda} w^{-1}(
\delta- \theta) = \delta- t_{-\lambda}\alpha= \bigl(1-\langle\lambda,\alpha
\rangle\bigr)\delta -\alpha\in R^+_{\af}
\]
since $\lambda$ is antidominant by Lemma~\ref{lem:Grass}. To show
that $s_0x$ is affine Grassmannian, we calculate for $\beta\in R^+$
\begin{eqnarray*}
r_\theta t_{-\theta^\vee} x (\beta) &= &r_\theta t_{-\theta^\vee
}
\bigl(w\beta- \langle\lambda,\beta \rangle\delta\bigr)
\\
&=& (r_\theta w) (\beta)+\bigl(\bigl\langle\theta^\vee,w\beta
\bigr\rangle -\langle\lambda,\beta \rangle\bigr)\delta.
\end{eqnarray*}
We need to show that the root $(r_\theta w)(\beta)
+(\langle\theta^\vee,w\beta \rangle -\langle\lambda,\beta
\rangle)\delta$ is positive.

First suppose that $\langle\lambda,\beta \rangle = 0$. Then by Lemma~\ref{lem:Grass}, we have $w \beta\in R^+$, so since $\theta$ is the
highest root we must have $\langle\theta^\vee,w \beta \rangle \geq
0$ by Lemma~\ref{lem:high}. If $\langle\theta^\vee,w\beta \rangle > 0$, we
are done. If
$\langle\theta^\vee,w\beta \rangle = 0$, we must show that
$(r_\theta w)\beta
\in R^+$. We calculate that $(r_\theta w)\beta= w r_\alpha\beta$.
But $\langle\alpha^\vee,\beta \rangle = \langle\theta^\vee,w\beta \rangle = 0$, so that
$w r_\alpha\beta= w \beta\in R^+$.

Now suppose $\langle\lambda,\beta \rangle < 0$. If $w\beta\in R^+$
then by Lemma~\ref{lem:high} we have $\langle\theta^\vee,w\beta \rangle \geq
0$, so we would
be done. If $w \beta\in R^-$, we note that $w \beta\neq-\theta$ so by
Lemma~\ref{lem:high} it suffices to assume that
$\langle\theta^\vee,w\beta \rangle = -1$ and show that $r_\theta w
\beta
\in R^+$. But $r_\theta w \beta= w r_\alpha\beta= w(\beta+\alpha) =
w\beta+ \theta\in R^+$.

For the converse, let us suppose that $\ell(r_\theta w) < \ell(w)$. Let
$\alpha
= -w^{-1}\theta\in R^+$. We have
\[
x^{-1} \alpha_0 = t_{-\lambda}w^{-1}(\delta-
\theta) = \alpha+ \bigl(1+\langle\lambda,\alpha \rangle\bigr)\delta.
\]
But $w \alpha= -\theta\in R^-$, so by Lemma~\ref{lem:Grass}, we have
$\langle\lambda,\alpha \rangle < 0$. If $\langle\lambda,\alpha
\rangle < -1$, then
$x^{-1}\alpha_0$ is a negative root, so that $s_0 x \prec x$. Otherwise,
we have $\langle\lambda,\alpha \rangle = -1$. In this case, we
calculate that
\[
(s_0 x) \alpha= (r_\theta t_{-\theta^\vee} w
t_\lambda) \alpha
\\
= (r_\theta w t_{\lambda+ \alpha^\vee})\alpha
\\
= r_\theta w \alpha- \bigl\langle\lambda+\alpha^\vee,\alpha
\bigr\rangle \delta.
\]
%
But $\langle\alpha^\vee,\alpha \rangle = 2$, so $(s_0x)\alpha\in
R^-_{\af
}$, and thus
$s_0x$ is not affine Grassmannian.
\end{pf}


\subsection{Proof of Theorem \texorpdfstring{\protect\ref{T:main}}{1}}
\label{ss:proof}
Let $Z = (Z_0,Z_1,\ldots)$ be a random walk on $\Theta_W$ with
transition matrix $P$, and $e = (w \to u)$ an edge in $\Theta_W$. Write
$\kappa_{e,N}(Z)$ for the number of times the edge $e$ is used in
$(Z_0,Z_1,\ldots,Z_N)$.

\begin{lemma}\label{lem:ergodic}
We have
\[
\lim_{N \to\infty}\frac{1}{N} \kappa_{e,N}(Z) =
\zeta(w)/r,
\]
almost surely.
\end{lemma}

\begin{pf}
This follows from the ergodic theorem for Markov chains; see for
example \cite{Bre}, Corollary~4.1.
\end{pf}

\begin{pf*}{Proof of Theorem~\ref{T:main} and first statement of
Theorem~\ref{T:main2}}
We first establish the statement for the delayed affine Grassmannian
random walk $(\tY_0,\tY_1,\ldots)$. Outside a set of measure 0 (those
$\tY$ that eventually stop), $\tY$ naturally maps (by removing repeats)
to the random walk $Y$ defined in Section~\ref{ss:random}.

Let the projection of $\tY$ to $W$ be $\pi(\tY) = Z$, which is a Markov
chain on $\Theta_W$ by Proposition~\ref{P:Markov}. Write $\tY_i =
A_{x_i}$, where $x_i=w_i t_{\lambda^{(i)}}$. The translation element
$\lambda^{(i)}$ only changes from $i$ to $i+1$ if $x_{i+1} = s_0 x_i$.
By Lemma~\ref{L:s0}, this corresponds to transitions $(w_i \to
r_\theta
w_i)$ in $Z$, which changes $\lambda^{(i)}$ by $w_i^{-1}(-\theta^\vee)$
(using $s_0 = r_\theta t_{-\theta^\vee}$).

For two edges $e,e'$, by Lemma~\ref{lem:ergodic}, the ratio $\frac
{\kappa_{e,N}(Z)}{\kappa_{e',N}(Z)}$ converges almost surely to
$\zeta
(w)/\zeta(w')$. It follows that
%
\begin{equation}
\lim_{N \to\infty} \sp\bigl(\lambda^{(N)}\bigr)
\to\sp \biggl(\sum_{w \in
W \dvtx \ell(r_\theta w) > \ell(w)} \zeta(w) w^{-1}
\bigl(-\theta^\vee\bigr) \biggr)
\end{equation}
almost surely. The alcove $A_{w_it_{\lambda^{(i)}}}$ shares a vertex
with the alcove $A_{t_{\lambda^{(i)}}}$, and so $-\lambda^{(i)}$ points
in almost the same direction as $v(\tY_i)$. We thus obtain Theorem~\ref
{T:main} and the first statement of Theorem~\ref{T:main2} for the
reduced affine Grassmannian random walk~$Y$.

Now, the random walk $X=(X_0,X_1,\ldots)$ will eventually stay in some
Weyl chamber, since each Weyl chamber is separated from the fundamental
alcove by some hyperplanes which can be crossed at most once, and there
are finitely many Weyl chambers.

The asymptotic direction of $Y$ does not depend on initial point of the
random walk, but only the constraint that the walk remains inside the
fundamental chamber and heads away from $A^\circ$. Thus, if we know
that $X \in C_w$, we can apply the right action of $W$ to the part of
$X$ lying inside $C_w$ to get a random walk in the fundamental chamber
which almost surely has asymptotic direction $\psi$, completing the proof.
\end{pf*}

The almost sure convergence of Theorem~\ref{T:main} implies convergence
in probability. Pick a norm on $V$.

\begin{cor}\label{C:main}
For each $\varepsilon> 0$ and $\delta> 0$, there is a $M = M(\varepsilon,\delta)$ so that
\[
\Prob\bigl(\bigl|v(Y_{N}) - \psi\bigr| \geq\varepsilon\bigr) < \delta
\]
for $N > M$.
\end{cor}

\begin{remark}
It follows from the proof of Theorem~\ref{T:main} that the point $\psi$
has rational coordinates, when written in terms of simple coroots. This
implies that there is a translation element of $\tW$ which points in
the same direction as $\psi$.
\end{remark}


\begin{remark}
In Theorem~\ref{T:main} and Corollary~\ref{C:main}, only the limiting
direction is discussed. The formula in Lemma~\ref{lem:Grass} for the
length $\ell(t_\lambda)$ of a translation element allows us to
calculate the speed that the random walk is traveling from the
fundamental alcove.
\end{remark}

We give an explicit conjecture for $\psi$ when $W = S_n$. In the next
result we treat $\rho$ as a point in $V$ by identifying $V$ and $V^*$
in the usual way.

\begin{conjecture}\label{conj:A}
For $W = S_n$, we have $\psi= \gamma\rho$ for some $\gamma> 0$.
\end{conjecture}

\begin{remark}\label{rem:othertypes}
Conjecture \ref{conj:A} does not hold as stated for other types. Define
$\{a_i \mid i \in I\}$ by $\theta= \sum_i a_i \alpha_i$, and set $a_0
= 1$. Now, weight the transitions corresponding to left multiplication
by $s_i$ by a factor of $a_i$. Then our computations suggest that
Conjecture \ref{conj:A} still holds for type $B_n$, and that it is
close to holding in other types. The coefficients $a_i$ here are
connected via affine Dynkin diagram duality to the coefficients
$a_i^\vee$ that we expected to see for reasons related to the topology
of the affine Grassmannian; see Section~\ref{ss:Grass}. The duality may
be an artifact of our choice of $Q^\vee$ instead of $Q$ for the
definition of an affine Weyl group.
\end{remark}


\section{The probability of eventually staying in a Weyl
chamber}\label{sec:W}
\subsection{Global reversal of the random walk on \texorpdfstring{${\hat{W}}$}{W}}\label{ss:probW}
Let $X = (X_0,X_1,\ldots)$ be the reduced random walk in $\tW$. Write
$X \in C_w$ for the event that $X$ eventually stays in the Weyl chamber
$C_w$. Write $X_N \in C_w^v$ if $X_N \in C_w$ and the type of $X_N$ is
$v$. We use the same notation for the delayed random walk~$\tX$.

The reverse of the random walks $X$ or $\tX$ is a very different
process to the original process. For example, $X$ can go in many
directions, at least at the beginning of the walk, but reversing $X$
gives a walk which heads toward the fundamental chamber. Thus, the next
result is very surprising. It relies on a very special feature of
Coxeter groups, namely the associativity of the Demazure product.

Let $K$ denote the affine $0$-Hecke algebra of $\tW$ (see \cite{LSS}),
with generators $\{T_i \mid i \in I \cup\{0\}\}$, a $\Z$-basis $\{T_x
\mid x \in\tW\}$ where $T_{\id} = 1$, satisfying the multiplication formulae
\[
T_i T_x = %
\cases{ T_{s_i x}, &\quad  $
\mbox{if $\ell(s_i x) > \ell(x)$,}$ \vspace *{2pt}
\cr
T_{x}, & \quad$\mbox{otherwise},$} %
\]
and also
\[
T_x T_i = %
\cases{ T_{xs_i}, &\quad $
\mbox{if $\ell(xs_i) > \ell(x)$,}$ \vspace *{2pt}
\cr
T_{x},
& \quad$\mbox{otherwise.}$} %
\]

In the following, we will freely identify alcoves with elements of $\tW$.

\begin{lemma}\label{L:reverse}
For each $x \in\tW$, we have $\Prob(\tX_N = x) = \Prob(\tX_N =
x^{-1})$, and $\Prob(X_N = x) = \Prob(X_N = x^{-1})$.
\end{lemma}

\begin{pf}
Let $\xi= \frac{1}{|I|+1}(\sum_{i \in I \cup\{0\}} T_i) \in K$. Then
$\Prob(\tX_N = x) = [T_x](\xi)^N$
where $[T_x]$ denotes the coefficient of $T_x$ when an element of $K$
is written in the basis $\{T_y \mid y \in\tW\}$. But the element $\xi$
of $K$ is invariant under the algebra antimorphism $T_x \mapsto
T_{x^{-1}}$ of $K$. It follows that the coefficient of $T_x$ and
$T_{x^{-1}}$ in the product $\xi^N$ coincides. Restricting to elements
with length $N$ gives the second statement.
\end{pf}

We call $x = wt_\lambda\in\tW$ regular if $\lambda\in Q^\vee$ is
regular, that is, the stabilizer subgroup of $W$ acting on $\lambda$
is trivial.

\begin{lemma}\label{L:reg}
Suppose $x \in C_w^v$ is regular. Then $x^{-1} \in C_{w_0wv^{-1}}^{v^{-1}}$.
\end{lemma}

\begin{pf}
If $x \in C_w^v$ is regular, then $x = v t_{w^{-1} \mu}$, where $\mu$
is a regular and antidominant. Then $x^{-1} = w^{-1} t_{-\mu} w v^{-1}
= w^{-1}w_0 t_{w_0(-\mu)} w_0 w v^{-1}$, and $w_0(-\mu)$ is antidominant.
\end{pf}


\begin{pf*}{Proof of second statement of Theorem~\ref{T:main2}}
It is clear that\break $\Prob(X \in C_w) = \Prob(\tX\in C_w)$, so we shall
focus on the delayed walk. Let $\eta(w) = \Prob(\tX\in C_w)$. In the
proof of Theorem~\ref{T:main}, we considered the delayed affine
Grassmannian walk $\tY$, or equivalently, a walk conditioned to lie in
$C_\mathrm{id}$. It follows from Proposition~\ref{P:Markov} that for such
a walk $\Prob(\tY\in C^v_{\id}) = \zeta(v)$. This same argument can be
applied to a walk conditioned to lie in any of the cones $C_w$, and we obtain
\[
\lim_{N \to\infty} \Prob\bigl(\tX_N \in
C^v_w\bigr) = \eta(w) \zeta\bigl(vw^{-1}\bigr).
\]
It follows from Theorem~\ref{T:main} that $\Prob(\tX_N \mbox{is
regular}) \to1$ as $N \to\infty$. Thus, using Lemmas \ref{L:reverse}
and \ref{L:reg}, for each $\varepsilon$ we can find $N$ sufficiently large
so that
\[
\bigl|\Prob\bigl(\tX_N \in C_w^v\bigr) - \Prob
\bigl(\tX_N \in C^{v^{-1}}_{w_0wv^{-1}}\bigr)\bigr| < \varepsilon.
\]
It follows that $\eta(w) \zeta(vw^{-1}) = \eta(w_0wv^{-1})\zeta
(w^{-1}w_0)$ for every $v,w \in W$. We note that setting $\eta(w) =
\zeta(w^{-1}w_0)$ solves this equation, and since $\eta$ is a
probability measure on $W$ this must be the solution.
\end{pf*}

\subsection{The Shi arrangement}\label{ss:Shi}
The ideas here are related to the language of reduced words in affine
Coxeter groups; see, for example, \cite{BB,Hea}.
The 
\defn{Shi arrangement} is the hyperplane arrangement consisting of
the hyperplanes $\{H_\alpha^0, H_\alpha^1 \mid\alpha\in R^+\}$. One
of the regions (connected components of the complement) of the Shi
arrangement is exactly the fundamental alcove $A^\circ$.

Let $B$ and $B'$ be two regions of the Shi arrangement. We say that $B$
is less than or equal to $B'$, and write $B \trianglelefteq B'$ if the
set of hyperplanes of the Shi arrangement separating $B'$ from the
fundamental alcove, contains the same set for $B$. 

Let $\Gamma$ denote the set of pairs $(B,w)$, where $B$ is a region of
the Shi arrangement, and $w \in W$ is such that $B$ contains an alcove
of type $w$. We make $\Gamma$ into a directed graph by defining edges
$(B,w) \to(B',w')$ whenever $B \trianglelefteq B'$, and an alcove $A$
of type $w$ in $B$ is adjacent (shares a facet) with an alcove $A'$ of
type $w'$ in $B'$, satisfying $A \lessdot A'$.

\begin{lemma}\label{L:Shi}If $(B,w) \to(B',w')$ then every alcove $A$
of type $w$ in $B$ shares a facet with an alcove $A'$ of type $w'$ in
$B'$, and we have $A \lessdot A'$.
\end{lemma}

\begin{pf}
Suppose $A$ and $\tA$ are both of type $w$ inside $B$. Set $\tA= A +
\lambda$. Let $H$ be a hyperplane (not necessarily belonging to the Shi
arrangement) cutting out a facet of (the closure of) $A$, and suppose
$A'$ is on the other side of $H$, adjacent to $A$ and satisfying $A
\lessdot A'$. Similarly, define $\tA'$ adjacent to $\tA$, on the other
side of $\tH:= H + \lambda$. Clearly, $\tA' = A' + \lambda$.

Since $A$ and $\tA$ belong to the same region of the Shi arrangement,
the line segment joining the center of $A$ to the center of $\tA$ does
not intersect the Shi arrangement. But one can go from $A'$ to $\tA'$
by crossing $H$, traveling from $A$ to $\tA$ and crossing $\tH$. Thus,
the only hyperplanes of the Shi arrangement that could separate $A'$
from $\tA'$ are the parallel hyperplanes $H$ and $\tH$.

Suppose first that $H$ belongs to the Shi arrangement. If at least one
of $H$ or $\tH$ separates $A'$ from $\tA'$, then since $A$ and $\tA$
are on the same side of $H$, it follows that $H$ separates $A'$ from
$\tA'$. We have that $\lambda$ cannot be parallel to $H$ (otherwise $H
= \tH$). Let $H$ be orthogonal to the root $\alpha$, so that we must
have $\langle\lambda,\alpha \rangle \neq0$. But from \eqref
{E:affineaction} it
is easy to see that one of the hyperplanes $H_\alpha^k$ was crossed
going from $A$ to $\tA$. It follows that the region $B$ is not bounded
in the $\alpha$ direction. The hyperplane $H$ must thus be $H_\alpha^0$
or $H_\alpha^1$. In either case, it separates $A$ from $A^\circ$,
contradicting the assumption $A \lessdot A'$.

So if $H$ belongs to the Shi arrangement, we conclude that $H = \tH$,
and that $\tA'$ and $A'$ belong to the same region $B'$ of the Shi
arrangement. Since $H$ separates $A'$ from $A^\circ$, and we also have
$\tA\lessdot\tA'$.

Finally, suppose that $H$ does not belong to the Shi arrangment. Then
$A, A', \tA,   \tA'$ all belong to the same region $B$, and are all
separated from $A^\circ$ by some $H_\alpha^0$ or $H_\alpha^1$ parallel
to $H$. In this case, the claim is clear.
\end{pf}

Denote by $B_w$ the unique region of the Shi arrangement that is a
translation of the Weyl chamber $C_w$. Let $\Gamma'$ be the graph
obtained from $\Gamma$ by removing $\{(B_v,u) \mid v,u \in W\}$. Let
$M$ be the transition matrix of $\Gamma$ and let $M'$ be its
restriction to $\Gamma'$. Let $\pp^w$ be the vector with components
labeled by vertices of $\Gamma'$, given by $\pp^w_{(B,v)}=\sum_{u
\in
W} \Prob((B,v) \to(B_w,u))$. Note that for each $(B,v)$, there is at
most one $u \in W$ for which the probability $\Prob((B,v) \to(B_w,u))$
is nonzero.

Let $\varepsilon_{(B,w)}$ denote the unit vector corresponding to a vertex
of $\Gamma$, and $\langle\cdot,\cdot \rangle$ denote the natural inner
product on the
vertex space spanned by vertices of~$\Gamma$.

\begin{thmm}\label{T:Shi}
For each $w \in W$,
\[
\zeta\bigl(w^{-1}w_0\bigr) = \Prob(X \in C_w)
= \bigl\langle\bigl(I-M'\bigr)^{-1} \cdot \varepsilon
_{(A^\circ,1)}, \pp^w \bigr\rangle.
\]
\end{thmm}

\begin{pf}
Lemma~\ref{L:Shi} guarantees that the Markov chain $X =
(X_0,X_1,\ldots
)$ projects to a Markov chain on $\Gamma$ via $x = vt_\lambda\mapsto
(B, v)$ where the alcove $A_x$ lies in the region $B$. Thus, the
probability $\Prob(X \in C_w)$ we desire is equal to the probability
that a random walk in $\Gamma$ starting from $(A^\circ,1)$, with
transition matrix $M$, eventually ends up at one of the vertices
$(B_w,v)$. This immediately gives the stated formula, assuming that
$(I-M')^{-1}$ is invertible, and is equal to $I + M' + (M')^2 + \cdots.$

Let $B$ be a region of the Shi arrangement which lies between two
parallel hyperplanes $H_\alpha^0$ and $H_\alpha^1$. Then for each $A
\in B$, there is some $A' \succ A$ outside of $B$. It follows that the
random walk $(X_0,X_1,\ldots)$ has probability 0 of staying in a region
of the Shi arrangement other than one of the $B_w$'s. Thus, $I-M'$ must
be invertible, $M'$ must be strictly substochastic, and $I + M' +
(M')^2 + \cdots= I-M'$.
\end{pf}

\begin{figure}

\includegraphics{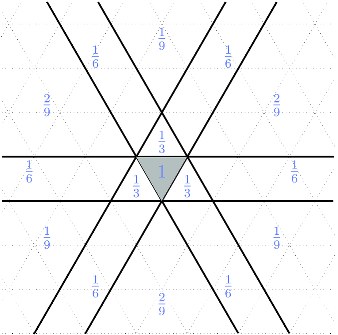}

\caption{Probabilities that $X$ passes through each region of the Shi
arrangement of $\tilde A_2$. The probabilities of the (translated) Weyl
chambers should be compared with Figure \protect\ref{fig:S3}, illustrating
Theorems \protect\ref{T:main2} and \protect\ref{T:Shi}.}
\label{fig:Shi}
\end{figure}

Theorem \ref{T:Shi} is illustrated in Figure \ref{fig:Shi}.

\section{$n$-cores, periodic TASEP and the connection to symmetric functions}
\label{sec:cores}
\subsection{$n$-cores and affine Grassmannian permutations}
In this section, we suppose $W = S_n$ is the symmetric group. We assume
basic familiarity with Young diagrams. Recall that a skew Young diagram
$\lambda/\mu$ is a ribbon if it is edge-connected and does not contain
any $2 \times2$ square. A Young diagram $\lambda$ is called an \textit{$n$-core} if no ribbons of size $n$ can be removed from it (and still
leaving a Young diagram).

The set of $n$-cores can be built from the empty partition by the
following procedure. Take an $n$-core $\lambda$, and suppose $b$ is an
addable-corner of $\lambda$ on diagonal $d$. Then the Young diagram
obtained from $\lambda$ by adding all addable-corners on diagonals $d'$
satisfying $d' \equiv d \mod n$, is also an $n$-core, and recursively
one obtains every $n$-core in this way.
Figure~\ref{fig:3} shows the start of the $3$-core graph, where the
edges denote the above box adding operation. The $3$-core graph is the
one-skeleton of a hexagonal planar tiling. The following result is well
known; see \cite{LLMS}.

\begin{figure}[b]

\includegraphics{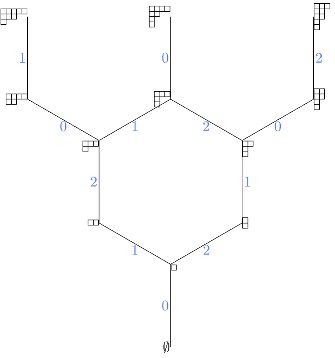}

\caption{The graph of $3$-cores, with edges labeled by the
corresponding simple generator. Note that $3$-cores on the same level
do not have the same number of
boxes.}
\label{fig:3}
\end{figure}

%

\begin{proposition}\label{P:bij}
There is a natural bijection between $n$-cores and the affine
Grassmannian elements of $\tS_n$. The edges of the $n$-core graph
correspond to left-multiplication by simple generators.
\end{proposition}

In the following, we use the standard coordinates for $Q^\vee$, so that
$\alpha_i^\vee= e_i - e_{i+1}$.

\begin{lemma}\label{L:slope}
Let $\mu= (\mu_1,\mu_2,\ldots,\mu_n) \in Q^\vee$ be an antidominant
element of the coroot lattice. Then the $n$-core of the translation
element $t_{(\mu_1,\mu_2,\ldots,\mu_n)}$ has slope $(n-i)/i$ between
diagonals $n\mu_i+i-2$ and $n \mu_{i+1} + i-2$, for $i = 1,2,\ldots,n-1$.\footnote{The slope should be calculated between the points of
intersection of the boundary of the core, and the diagonals, but for
our asymptotic purposes this is not important.}
\end{lemma}

\begin{pf}
Follows from \cite{LLMS}, Proposition~8.10.
\end{pf}

The $4$-core in Figure~\ref{fig:core} corresponds to $(-7,-2,3,6) \in
Q^\vee$.

%

\subsection{The shape of a random $n$-core}\label{sec:randomcurve}

By a random $n$-core we will mean an $n$-core generated by applying the
bijection in Proposition~\ref{P:bij} to the Markov chain $Y$ described
in Section~\ref{ss:random}. If $\lambda$ is a $n$-core, then we let
$D(\lambda)$ denote the curve drawing out the lower-right boundary of
$\lambda$, scaled by the degree $\deg(\lambda)$ in both directions.
Here, the degree is the length of the corresponding affine Grassmannian
element from Proposition~\ref{P:bij}, or equivalently, the distance
from the empty partition in the $n$-core graph. By convention,
$D(\lambda)$ includes a vertical ray going to $-\infty$ along the
$y$-axis, and a horizontal ray going to $+\infty$ along the $x$-axis.
Given two curves $D, D'$ of this form, we write $|D-D'|$ to denote the
supremum of the distance between $D$ and $D'$, measured along the
diagonals $y = -x + k$. With this notation, Corollary~\ref{C:main}
combined with Lemma~\ref{L:slope} translates to Theorem~\ref{T:cores}.


Let us use Conjecture \ref{conj:A} to predict the piecewise-linear
curve $C_n$ of Theorem~\ref{T:cores}. Let $\mu$ be an antidominant
element of $Q^\vee$ satisfying $\mu_2-\mu_1 = \mu_3 - \mu_2 =
\cdots=
\mu_n - \mu_{n-1} = A$ (i.e., $\mu$ is in the same direction as
$\rho
$). To calculate the correct scaling we use Lemma~\ref{lem:Grass} which
says that
$ \ell(t_\mu) = \sum_{1\leq i <j \leq n} \mu_j - \mu_i = A/\alpha$,
where $\alpha= \frac{6}{(n-1)n(n+1)}$.

Now consider the piecewise-linear curve $C_\rho$ which successively
connects the points
\begin{eqnarray*}
&&(0,-\infty),\qquad\biggl(0,-\frac{n(n-1)}{2} \alpha\biggr), \qquad\bigl(\alpha, -(1+2+
\cdots +n-2)\alpha\bigr), \\
&&\qquad\bigl((1+2)\alpha, -(1+2+\cdots+n-3)\alpha\bigr), \qquad\ldots,
\\
&&\qquad\bigl((1+2+\cdots+n-2)\alpha,-\alpha\bigr),\qquad\biggl(\frac{n(n-1)}{2} \alpha,0
\biggr),\qquad (\infty,0).
\end{eqnarray*}
Using Lemma~\ref{L:slope}, one calculates that the core $\lambda$
corresponding to $t_{\mu}$ has diagram $D(\lambda)$ extremely close to
$C_n$: namely, it passes through the specified points but may not be
linear in between those points. Thus, we have the following proposition.

\begin{proposition}
Assuming Conjecture \ref{conj:A}, the curve $C_n$ of Theorem~\ref
{T:cores} is~$C_\rho$.
\end{proposition}

This proposition allows us to make some predictions, for example, of
the length of the first row of a random $n$-core. This might be
compared to corresponding results for random partitions (see, e.g.,
\cite{LoSh,VK}).

\begin{cor}
Assuming Conjecture \ref{conj:A}, the expected length of the first row
of a random $n$-core of degree $d$ is asymptotic to $\frac{3d}{n+1}$.
\end{cor}

%

The area between $C_\rho$ and the axes is equal to
\[
\operatorname{area}(C_\rho) = \frac{1}{2} \alpha^2\bigl ((n-1)^2 + 2(n-2)^2 + \cdots
+ (n-1)1^2\bigr) = \frac{n^2(n^2-1)\alpha^2}{24}.
\]
If we scale the limit shape so that this area is normalized to $1$,
then the $x$-intercept of $C_\rho$ would become $\frac{\sqrt {6}(n-1)}{\sqrt{n^2-1}}$.

\begin{cor}\label{cor:firstrow}
Assuming Conjecture \ref{conj:A}, the first row of a large random
$n$-core is asymptotic to $\frac{\sqrt{6}(n-1)}{\sqrt{n^2-1}}\sqrt{N}$,
where $N$ is the number of boxes in the $n$-core.
\end{cor}

\subsection{Periodic TASEP}
There is a well-known correspondence between growth models on Young
diagrams, and the totally asymmetric exclusion process (TASEP). The
random growth model on $n$-cores we have described gives rise to a
periodic analogue of TASEP that we now describe.

Let $\sigma= (\sigma_i \in\{0,1\} \mid i \in\Z)$ be a doubly
infinite sequence of $0$-s and $1$-s, labeled by the integers. The
sequence $\sigma$ is to be thought of as a sequence of balls and empty
spaces: $\sigma_i = 0$ mean that position $i$ is empty, and $\sigma_i =
1$ means that position $i$ is occupied. There is a natural map $\lambda
\mapsto\sigma(\lambda)$, illustrated in Figure~\ref{fig:TASEP}. The
indexing is normalized so that $\sigma(\varnothing)$ is the step-function
satisfying $\sigma_i = 1$ for $i < 0$ and $\sigma_i = 0$ for $i \geq
0$. It is clear that adding a box to $\lambda$ corresponds to moving a
ball to an empty space immediately to its right.

\begin{figure}

\includegraphics{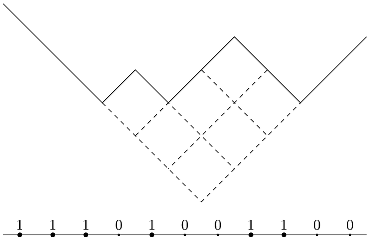}

\caption{The calculation of $\sigma((3,2,2))$ is illustrated here: we
first rotate the Young diagram 135 degrees counterclockwise, then we
draw the outline curve (illustrated in solid lines). Downward steps in
the outline curve corresponds to $1$-s, and upward steps correspond to $0$-s.}
\label{fig:TASEP}
\end{figure}

%

Suppose $\lambda$ is an $n$-core. Then $\sigma(\lambda)$ satisfies:
\begin{longlist}[(1)]
\item[(1)]
$\sigma_i = 1$ for $i \ll0$,
\item[(2)]
if $\sigma_i = 1$ then $\sigma_{i-n} = 1$ and
\item[(3)]
if $d_1,d_2,\ldots,d_n \in\Z$ are such that $\sigma_{d_j} = 1$ and
$\sigma_{d_j + n} = 0$, then we have $\sum_{j=1}^n d_j = -
{n+1\choose2}$,
\end{longlist}
and these conditions characterize the sequences that arise from
$n$-cores. Periodic TASEP is a random process on these sequences, given
by the rules:
\begin{longlist}
\item[(1)]
At time $t = 0$, we have $\sigma(0)$ is the step-function.
\item[(2)]
At each time $t$, an element $\bar i \in\Z/n\Z$ is chosen uniformly at
random, subject to the condition that there exists $i_0 \equiv\bar i
\mod n$ satisfying $\sigma_{i_0} = 1$ and $\sigma_{i_0+1} = 0$. We then
define $\sigma(t+1)$ by moving all balls at positions $i \equiv\bar i
\mod n$ one step to the right, if possible.
\end{longlist}
The conditioning implies that at each time step finitely many, but
nonzero number of balls are moved.

\begin{proposition}
The random $n$-core process is transformed under $\lambda\mapsto
\sigma
(\lambda)$ to the periodic TASEP process.
\end{proposition}

When $n = \infty$, periodic TASEP becomes one of the standard discrete
time versions of the TASEP process. Namely, at each time $t$, one of
the balls that can be moved is chosen uniformly at random, and moved
one step to the right. The asymptotic behavior of TASEP is a very
well-studied problem. In particular, Rost~\cite{Ros} (see also
Johansson \cite{Joh}) has described the asymptotic shape of the result.
%
\begin{figure}

\includegraphics{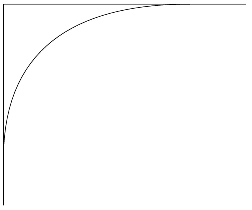}

\caption{The limiting curve $C$ for TASEP.}
\label{fig:TASEPC}
\end{figure}

We describe their result in terms of Young diagrams, and also rotated
so that Young diagrams are upper-left justified. As $t \to\infty$, the
Young diagram of this growth process, after suitable scaling,
approaches the limiting curve (see Figure~\ref{fig:TASEPC})
\begin{eqnarray*}
C& =& \bigl\{(x,0) \mid x \in[1,\infty)\bigr\} \cup\bigl\{(x,y) \in[0,1] \times
[-1,0] \mid\sqrt{x} + \sqrt{-y} = 1\bigr\}\\
&&{} \cup\bigl\{(y,0) \mid y \in [-1,-
\infty)\bigr\}.
\end{eqnarray*}

%

It is not hard to see that after a suitable scaling, the
piecewise-linear curves $C_\rho$ of Section~\ref{sec:randomcurve}
approaches $C$ pointwise, as $n \to\infty$.

\subsection{Plancherel measure for $n$-cores}\label{ss:plancherel}
This work was motivated by the connections to a family $\tF_x(X)$ of
symmetric functions labeled by $x \in\tS_n$, known as \emph{affine
Stanley symmetric functions} \cite{LamAS} (and also a closely related
family $\tG_x(X)$ called the affine stable Grothendieck polynomials
\cite{LSS}). The coefficient $[m_{1^{\ell(x)}}] \tF_x(X)$ of the
square-free monomial in $\tF_x$ is equal to the number of reduced words
of~$x$. Whereas Stanley's seminal work \cite{Sta} studies \textit{exact}
formulae for the number of reduced words, our approach looks for \textit{asymptotic} formulae. The symmetric functions $\tF_x$ plays the same
role for affine permutations, namely, a generating function for
``semi-standard'' objects, as the Schur functions $s_\lambda$ play for
Grassmannian permutations. Schur functions play a crucial role in the
study of random partitions; see, for example, \cite{Oko}.

The measure we obtain on the set $\{x \in\tS_n \mid\ell(x) = N\}$ of
affine permutations of length $N$ from our random walk is not the same
measure as the one obtained by letting $\Prob(x)$ be proportional to
the number of reduced words of $x$. Nevertheless, Corollary~\ref
{C:main} and Theorem~\ref{T:cores} still apply (see Remark~\ref
{rem:reduced_words}).

In \cite{LLMS}, we proved an enumerative identity
%
\begin{equation}
\label{E:nfact} m! = \sum_\lambda\#\{\mbox{weak tableaux
of shape $\lambda$}\} \cdot\# \{\mbox{strong tableaux of shape $\lambda$}\},
\end{equation}
where the sum is over $n$-cores of degree $m$. \textit{Weak tableaux}
count paths in the $n$-core graph. \textit{Strong tableaux} are defined in
terms of the strong (Bruhat) order. The terms on the right-hand side of
\eqref{E:nfact} would give the natural analogue of the Plancherel
measure for partitions. In \cite{LLMS}, a symmetric function
generalization of \eqref{E:nfact} is also given, and involves affine
Stanley symmetric functions and {\it$k$-Schur functions}. The identity
\eqref{E:nfact} is generalized to the Kac--Moody case in \cite{LS}.

\subsection{$K$-homology of the affine Grassmannian}\label{ss:Grass}
Recall from the proof of Lemma~\ref{L:reverse} in Section~\ref
{ss:probW} that the probabilities $\Prob(X_N = x)$ were given by the
coefficients $[T_x] \xi^N$ for an element $\xi\in K$. In the case $W =
S_n$, by \cite{LSS}, Corollary~7.5, the element $\xi$ can be
interpreted (up to a factor) as the divisor Schubert class in the
$K$-homology $K^*(\Gr_{\mathrm{SL}(n)})$ of the affine Grassmannian of $\mathrm{SL}(n)$.
The affine Grassmannian $\Gr_{\mathrm{SL}(n)}$ is an infinite-dimensional space
of central importance in representation theory. In the case of a
complex simple algebraic group with Weyl group $W$, the natural element
to consider from the point of view of the geometry of $\Gr_{G}$ is
\[
\xi' = \sum_{i=I \cup\{0\}} a_i^\vee
T_i,
\]
where the definition of the weights $a_i^\vee$ can be found in \cite
{Kac}; see \cite{LS}, Proposition~2.17, for an explanation of these
weights (the argument in \cite{LS} is for the homology case, but easily
extends to $K$-homology). Probabilistically, this amounts to
considering random walks where the allowable transitions are not taken
uniformly at random, but left multiplication by $s_i$ is weighted by
the $a_i^\vee$. Note that Theorem~\ref{T:main} and its proof still
remain valid in this situation.
See also Remark~\ref{rem:othertypes}.

%
\section*{Acknowledgements}
We benefited from John Stembridge's \texttt{coxeter/weyl} Maple package.
We thank Jinho Baik and Alexei Borodin for helpful conversations
related to the totally assymetric exclusion process. We also thank an
anonymous referee for helpful suggestions.
We thank Arvind Ayyer and Svante Linusson for a correction.

%

%




\printaddresses


\begin{thebibliography}{27}

%
%
\bibitem{Aas}
\begin{bmisc}[auto:STB|2014/02/12|14:17:21]
\bauthor{\bsnm{Aas},~\bfnm{E.}\binits{E.}}
(\byear{2012}).
\bhowpublished{Stationary probability of the identity for the TASEP on
a ring.
Preprint. Available at \arxivurl{arXiv:1212.6366}.}
\end{bmisc}
\bptok{imsref}\endbibitem

%
%
\bibitem{AL2}
\begin{bmisc}[auto:STB|2014/02/12|14:17:21]
\bauthor{\bsnm{Ayyer},~\bfnm{A.}\binits{A.}} \AND
\bauthor{\bsnm{Linusson},~\bfnm{S.}\binits{S.}}
(\byear{2014}).
\bhowpublished{Correlations in the multispecies TASEP and a
conjecture by Lam. Preprint.}
\end{bmisc}
\bptok{imsref}\endbibitem


%
%
\bibitem{AL}
\begin{barticle}[auto:STB|2014/02/12|14:17:21]
\bauthor{\bsnm{Ayyer},~\bfnm{A.}\binits{A.}} \AND
\bauthor{\bsnm{Linusson},~\bfnm{S.}\binits{S.}}
(\byear{2015}).
\btitle{An inhomogeneous multispecies TASEP on a ring}.
\bjournal{Adv. Appl. Math.}
\bvolume{57}
\bpages{21--43}.
\bid{mr={3206520}}
\end{barticle}
\bptok{imsref}\endbibitem


%
%
\bibitem{BHR}
\begin{barticle}[mr]
\bauthor{\bsnm{Bidigare},~\bfnm{Pat}\binits{P.}},
\bauthor{\bsnm{Hanlon},~\bfnm{Phil}\binits{P.}} \AND
\bauthor{\bsnm{Rockmore},~\bfnm{Dan}\binits{D.}}
(\byear{1999}).
\btitle{A combinatorial description of the spectrum for the {T}setlin
library and its generalization to hyperplane arrangements}.
\bjournal{Duke Math. J.}
\bvolume{99}
\bpages{135--174}.
\bid{issn={0012-7094}, mr={1700744}}
\end{barticle}
\bptok{imsref}
\endbibitem

%
%
\bibitem{BB}
\begin{bbook}[mr]
\bauthor{\bsnm{Bj{\"o}rner},~\bfnm{Anders}\binits{A.}} \AND
\bauthor{\bsnm{Brenti},~\bfnm{Francesco}\binits{F.}}
(\byear{2005}).
\btitle{Combinatorics of {C}oxeter Groups}.
\bseries{Graduate Texts in Mathematics}
\bvolume{231}.
\bpublisher{Springer},
\blocation{New York}.
\bid{mr={2133266}}
\end{bbook}
\bptok{imsref}\endbibitem

%
%
\bibitem{Bre}
\begin{bbook}[mr]
\bauthor{\bsnm{Br{\'e}maud},~\bfnm{Pierre}\binits{P.}}
(\byear{1999}).
\btitle{Markov Chains: Gibbs Fields, Monte Carlo Simulation, and Queues}.
\bseries{Texts in Applied Mathematics}
\bvolume{31}.
\bpublisher{Springer},
\blocation{New York}.
\bid{mr={1689633}}
\end{bbook}
\bptok{imsref}
\endbibitem

%
%
\bibitem{BD}
\begin{barticle}[mr]
\bauthor{\bsnm{Brown},~\bfnm{Kenneth~S.}\binits{K.~S.}} \AND
\bauthor{\bsnm{Diaconis},~\bfnm{Persi}\binits{P.}}
(\byear{1998}).
\btitle{Random walks and hyperplane arrangements}.
\bjournal{Ann. Probab.}
\bvolume{26}
\bpages{1813--1854}.
\bid{issn={0091-1798}, mr={1675083}}
\end{barticle}
\bptok{imsref}
\endbibitem

%
%
\bibitem{FM}
\begin{barticle}[mr]
\bauthor{\bsnm{Ferrari},~\bfnm{Pablo~A.}\binits{P.~A.}} \AND
\bauthor{\bsnm{Martin},~\bfnm{James~B.}\binits{J.~B.}}
(\byear{2007}).
\btitle{Stationary distributions of multi-type totally asymmetric
exclusion processes}.
\bjournal{Ann. Probab.}
\bvolume{35}
\bpages{807--832}.
\bid{issn={0091-1798}, mr={2319708}}
\end{barticle}
\bptok{imsref}
\endbibitem

%
%
\bibitem{Hea}
\begin{bmisc}[mr]
\bauthor{\bsnm{Headley},~\bfnm{Patrick~Thomas}\binits{P.~T.}}
(\byear{1994}).
\bhowpublished{Reduced expressions in infinite {C}oxeter groups.
ProQuest LLC, Ann Arbor, MI, Ph.D. Thesis, Univ. Michigan}.
\bid{mr={2691313}}
\end{bmisc}
\bptok{imsref}
\endbibitem

%
%
\bibitem{HST}
\begin{barticle}[mr]
\bauthor{\bsnm{Hivert},~\bfnm{Florent}\binits{F.}},
\bauthor{\bsnm{Schilling},~\bfnm{Anne}\binits{A.}} \AND
\bauthor{\bsnm{Thi{\'e}ry},~\bfnm{Nicolas~M.}\binits{N.~M.}}
(\byear{2009}).
\btitle{Hecke group algebras as quotients of affine {H}ecke algebras
at level 0}.
\bjournal{J. Combin. Theory Ser. A}
\bvolume{116}
\bpages{844--863}.
\bid{issn={0097-3165}, mr={2513638}}
\end{barticle}
\bptok{imsref}
\endbibitem

%
%
\bibitem{Hum}
\begin{bbook}[mr]
\bauthor{\bsnm{Humphreys},~\bfnm{James~E.}\binits{J.~E.}}
(\byear{1990}).
\btitle{Reflection Groups and {C}oxeter Groups}.
\bseries{Cambridge Studies in Advanced Mathematics}
\bvolume{29}.
\bpublisher{Cambridge Univ. Press},
\blocation{Cambridge}.
\bid{mr={1066460}}
\end{bbook}
\bptok{imsref}
\endbibitem

%
%
\bibitem{Ito}
\begin{barticle}[mr]
\bauthor{\bsnm{Ito},~\bfnm{Ken}\binits{K.}}
(\byear{2005}).
\btitle{Parameterizations of infinite biconvex sets in affine root systems}.
\bjournal{Hiroshima Math. J.}
\bvolume{35}
\bpages{425--451}.
\bid{issn={0018-2079}, mr={2210718}}
\end{barticle}
\bptok{imsref}
\endbibitem

%
%
\bibitem{Joh}
\begin{barticle}[mr]
\bauthor{\bsnm{Johansson},~\bfnm{Kurt}\binits{K.}}
(\byear{2000}).
\btitle{Shape fluctuations and random matrices}.
\bjournal{Comm. Math. Phys.}
\bvolume{209}
\bpages{437--476}.
\bid{issn={0010-3616}, mr={1737991}}
\end{barticle}
\bptok{imsref}
\endbibitem

%
%
\bibitem{Kac}
\begin{bbook}[mr]
\bauthor{\bsnm{Kac},~\bfnm{Victor~G.}\binits{V.~G.}}
(\byear{1990}).
\btitle{Infinite-Dimensional {L}ie Algebras},
\bedition{3rd} ed.
\bpublisher{Cambridge Univ. Press},
\blocation{Cambridge}.
\bid{mr={1104219}}
\end{bbook}
\bptok{imsref}
\endbibitem

%
%
\bibitem{LamAS}
\begin{barticle}[mr]
\bauthor{\bsnm{Lam},~\bfnm{Thomas}\binits{T.}}
(\byear{2006}).
\btitle{Affine {S}tanley symmetric functions}.
\bjournal{Amer. J. Math.}
\bvolume{128}
\bpages{1553--1586}.
\bid{issn={0002-9327}, mr={2275911}}
\end{barticle}
\bptok{imsref}
\endbibitem

%
%
\bibitem{LamSP}
\begin{barticle}[mr]
\bauthor{\bsnm{Lam},~\bfnm{Thomas}\binits{T.}}
(\byear{2008}).
\btitle{Schubert polynomials for the affine {G}rassmannian}.
\bjournal{J. Amer. Math. Soc.}
\bvolume{21}
\bpages{259--281}.
\bid{issn={0894-0347}, mr={2350056}}
\end{barticle}
\bptok{imsref}
\endbibitem

%
%
\bibitem{LLMS}
\begin{barticle}[mr]
\bauthor{\bsnm{Lam},~\bfnm{Thomas}\binits{T.}},
\bauthor{\bsnm{Lapointe},~\bfnm{Luc}\binits{L.}},
\bauthor{\bsnm{Morse},~\bfnm{Jennifer}\binits{J.}} \AND
\bauthor{\bsnm{Shimozono},~\bfnm{Mark}\binits{M.}}
(\byear{2010}).
\btitle{Affine insertion and {P}ieri rules for the affine {G}rassmannian}.
\bjournal{Mem. Amer. Math. Soc.}
\bvolume{208}
\bpages{xii+82}.
\bid{issn={0065-9266}, mr={2741963}}
\end{barticle}
\bptok{imsref}
\endbibitem

%
%
\bibitem{LP}
\begin{barticle}[mr]
\bauthor{\bsnm{Lam},~\bfnm{Thomas}\binits{T.}} \AND
\bauthor{\bsnm{Pylyavskyy},~\bfnm{Pavlo}\binits{P.}}
(\byear{2013}).
\btitle{Total positivity for loop groups {II}: {C}hevalley generators}.
\bjournal{Transform. Groups}
\bvolume{18}
\bpages{179--231}.
\bid{issn={1083-4362}, mr={3022763}}
\end{barticle}
\bptok{imsref}
\endbibitem

%
%
\bibitem{LSS}
\begin{barticle}[mr]
\bauthor{\bsnm{Lam},~\bfnm{Thomas}\binits{T.}},
\bauthor{\bsnm{Schilling},~\bfnm{Anne}\binits{A.}} \AND
\bauthor{\bsnm{Shimozono},~\bfnm{Mark}\binits{M.}}
(\byear{2010}).
\btitle{{$K$}-theory {S}chubert calculus of the affine {G}rassmannian}.
\bjournal{Compos. Math.}
\bvolume{146}
\bpages{811--852}.
\bid{issn={0010-437X}, mr={2660675}}
\end{barticle}
\bptok{imsref}
\endbibitem

%
%
\bibitem{LW}
\begin{barticle}[mr]
\bauthor{\bsnm{Lam},~\bfnm{Thomas}\binits{T.}} \AND
\bauthor{\bsnm{Williams},~\bfnm{Lauren}\binits{L.}}
(\byear{2012}).
\btitle{A {M}arkov chain on the symmetric group that is {S}chubert positive?}
\bjournal{Experiment. Math.}
\bvolume{21}
\bpages{189--192}.
\bid{issn={1058-6458}, mr={2931313}}
\end{barticle}
\bptok{imsref}
\endbibitem

%
%
\bibitem{LS}
\begin{barticle}[mr]
\bauthor{\bsnm{Lam},~\bfnm{Thomas~F.}\binits{T.~F.}} \AND
\bauthor{\bsnm{Shimozono},~\bfnm{Mark}\binits{M.}}
(\byear{2007}).
\btitle{Dual graded graphs for {K}ac--{M}oody algebras}.
\bjournal{Algebra Number Theory}
\bvolume{1}
\bpages{451--488}.
\bid{issn={1937-0652}, mr={2368957}}
\end{barticle}
\bptok{imsref}
\endbibitem

%
%
\bibitem{LM}
\begin{bmisc}[auto:STB|2014/02/12|14:17:21]
\bauthor{\bsnm{Linusson},~\bfnm{S.}\binits{S.}} \AND
\bauthor{\bsnm{Martin},~\bfnm{J.~B.}\binits{J.~B.}}
\bhowpublished{Stationary probabilities for an inhomogeneous multi-type TASEP.
In preparation.}
\end{bmisc}
\bptok{imsref}\endbibitem

%
%
\bibitem{LoSh}
\begin{barticle}[mr]
\bauthor{\bsnm{Logan},~\bfnm{B.~F.}\binits{B.~F.}} \AND
\bauthor{\bsnm{Shepp},~\bfnm{L.~A.}\binits{L.~A.}}
(\byear{1977}).
\btitle{A variational problem for random {Y}oung tableaux}.
\bjournal{Adv. Math.}
\bvolume{26}
\bpages{206--222}.
\bid{mr={1417317}}
\end{barticle}
\bptok{imsref}
\endbibitem

%
%
\bibitem{Oko}
\begin{barticle}[mr]
\bauthor{\bsnm{Okounkov},~\bfnm{Andrei}\binits{A.}}
(\byear{2001}).
\btitle{Infinite wedge and random partitions}.
\bjournal{Selecta Math. (N.S.)}
\bvolume{7}
\bpages{57--81}.
\bid{issn={1022-1824}, mr={1856553}}
\end{barticle}
\bptok{imsref}
\endbibitem

%
%
\bibitem{Ros}
\begin{barticle}[mr]
\bauthor{\bsnm{Rost},~\bfnm{H.}\binits{H.}}
(\byear{1981}).
\btitle{Nonequilibrium behaviour of a many particle process: Density
profile and local equilibria}.
\bjournal{Z. Wahrsch. Verw. Gebiete}
\bvolume{58}
\bpages{41--53}.
\bid{issn={0044-3719}, mr={0635270}}
\end{barticle}
\bptok{imsref}
\endbibitem

%
%
\bibitem{Sta}
\begin{barticle}[mr]
\bauthor{\bsnm{Stanley},~\bfnm{Richard~P.}\binits{R.~P.}}
(\byear{1984}).
\btitle{On the number of reduced decompositions of elements of
{C}oxeter groups}.
\bjournal{European J. Combin.}
\bvolume{5}
\bpages{359--372}.
\bid{issn={0195-6698}, mr={0782057}}
\end{barticle}
\bptok{imsref}
\endbibitem

%
%
\bibitem{VK}
\begin{barticle}[mr]
\bauthor{\bsnm{Ver{\v{s}}ik},~\bfnm{A.~M.}\binits{A.~M.}} \AND
\bauthor{\bsnm{Kerov},~\bfnm{S.~V.}\binits{S.~V.}}
(\byear{1977}).
\btitle{Asymptotic behavior of the {P}lancherel measure of the
symmetric group and the limit form of {Y}oung tableaux}.
\bjournal{Dokl. Akad. Nauk SSSR}
\bvolume{233}
\bpages{1024--1027}.
\bid{issn={0002-3264}, mr={0480398}}
\end{barticle}
\bptok{imsref}
\endbibitem

\end{thebibliography}
\end{document}